\documentclass[11pt,a4paper]{amsart}
\usepackage{graphicx,amssymb}
\usepackage[top=25mm,bottom=20mm,left=30mm,right=30mm]{geometry}

\usepackage{xypic}
\xyoption{all}
\xyoption{arc}
\xyoption{poly}

\theoremstyle{plain}
\newtheorem{theorem}{Theorem}[section]
\newtheorem{lemma}[theorem]{Lemma}
\newtheorem{corollary}[theorem]{Corollary}
\newtheorem{proposition}[theorem]{Proposition}

\theoremstyle{definition}
\newtheorem{definition}[theorem]{Definition}
\newtheorem{example}[theorem]{Example}

\newenvironment{claim}[1]
{\par\addvspace{\medskipamount}\noindent\emph{#1.}\em\ \ignorespaces}
{\normalfont\par\addvspace{\medskipamount}}

\begin{document}

\title{Noncrossing partitions for periodic braids}

\author{Eon-Kyung Lee and Sang-Jin Lee}
\address{Department of Mathematics, Sejong University, Seoul, Korea}
\email{eonkyung@sejong.ac.kr}

\address{Department of Mathematics, Konkuk University, Seoul, Korea; 
School of Mathematics, Korea Institute for Advanced Study, Seoul, Korea}
\email{sangjin@konkuk.ac.kr}
\date{\today}

\begin{abstract}
An element in Artin's braid group $B_n$ is called periodic
if it has a power that lies in the center of $B_n$.
The conjugacy problem for periodic braids can be reduced
to the following: given a divisor $1\leqslant d<n-1$ of $n-1$
and an element $\alpha$ in the super summit set of $\varepsilon^d$,
find $\gamma\in B_n$ such that $\gamma^{-1}\alpha\gamma=\varepsilon^d$,
where $\varepsilon=(\sigma_{n-1}\cdots\sigma_1)\sigma_1$.

In this article we characterize the elements
in the super summit set of $\varepsilon^d$
in the dual Garside structure
by studying the combinatorics of noncrossing partitions
arising from periodic braids.
Our characterization directly provides a conjugating element $\gamma$.
And it determines the size of the super summit set of $\varepsilon^d$
by using the zeta polynomial of the noncrossing partition lattice.

\end{abstract}

\maketitle

\section{Introduction}

The conjugacy problem in a group has two versions:
the conjugacy decision problem (CDP) is to decide
whether two given elements are conjugate or not;
the conjugacy search problem (CSP) is to find a conjugating element
for a given pair of conjugate elements.
The conjugacy problem is of great interest
for Artin's braid group $B_n$,
which has the well-known presentation~\cite{Art25}:
\begin{equation}\label{eq:ClassPres}
B_n  =  \left\langle \sigma_1,\ldots,\sigma_{n-1} \left|
\begin{array}{ll}
\sigma_i \sigma_j = \sigma_j \sigma_i & \mbox{if } |i-j| \geqslant 2, \\
\sigma_i \sigma_j \sigma_i = \sigma_j \sigma_i \sigma_j & \mbox{if } |i-j| = 1.
\end{array}
\right.\right\rangle.
\end{equation}

The standard solutions to the conjugacy problem
in $B_n$ use Garside structures.
The braid group $B_n$ admits two standard
Garside structures, the {\em classical Garside structure}~\cite{Gar69,EC+92,EM94}
arising from Artin's presentation in~(\ref{eq:ClassPres}),
and the {\em dual Garside structure}~\cite{BKL98}.
Both structures provide efficient solutions to the word problem.

For each element $\alpha\in B_n$,
there is a \emph{super summit set} $[\alpha]^S$,
which depends on the choice of a particular Garside structure.
Every super summit set is finite and nonempty.
Two braids are conjugate
if and only if their super summit sets are the same.
Given a braid, one can compute an element in its super summit set
in polynomial time.
However, super summit sets are exponentially large at least in the braid index.

Let $\Delta = \sigma_1(\sigma_2\sigma_1)
\cdots(\sigma_{n-1}\cdots\sigma_2\sigma_1)$,
$\delta = \sigma_{n-1}\sigma_{n-2}\cdots\sigma_1$ and
$\varepsilon = \delta\sigma_1$.
See Figure~\ref{fig:circ}.
The center of $B_n$ is the cyclic group generated
by $\Delta^2=\delta^n=\varepsilon^{n-1}$.
An element of $B_n$ is called \emph{periodic} if some power of it
lies in the center $\langle\Delta^2\rangle$.
The results of Brouwer~\cite{Bro19}, Ker\'ekj\'art\'o~\cite{Ker19} and
Eilenberg~\cite{Eil34} imply that an $n$-braid is periodic if and only if
it is conjugate to a power of either $\delta$ or $\varepsilon$.

\begin{figure}
\begin{tabular}{ccc}
\includegraphics[scale=.65]{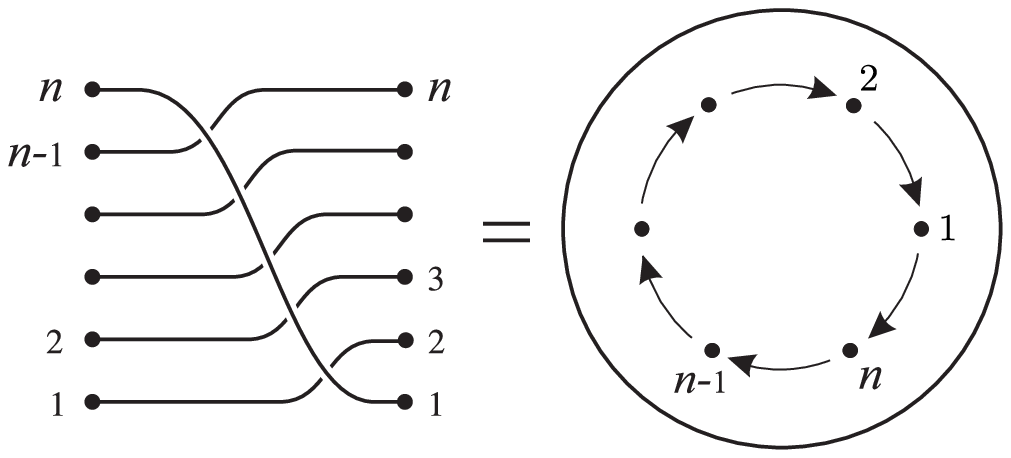}
& &
\includegraphics[scale=.65]{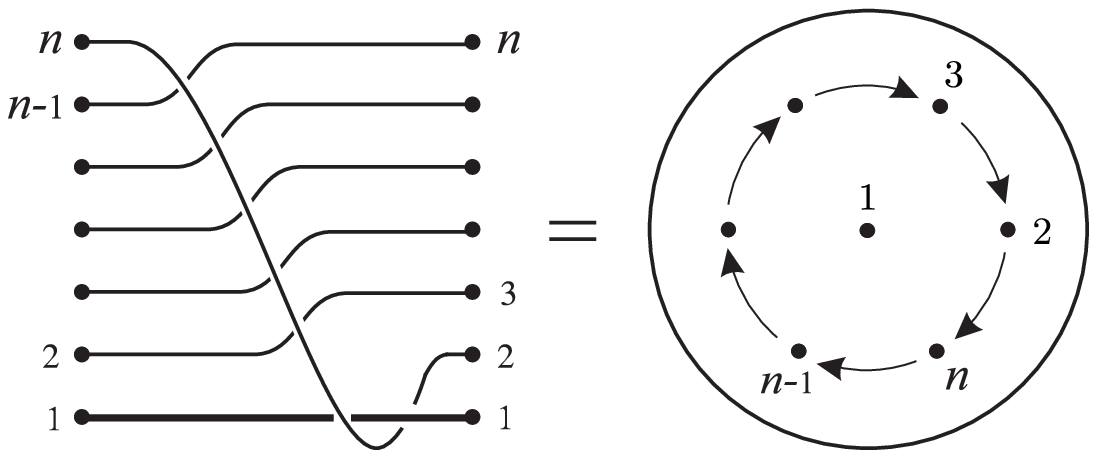}\\
(a) $\delta = \sigma_{n-1}\sigma_{n-2}\cdots\sigma_1$
&\qquad&
(b) $\varepsilon = \delta\sigma_1$
\end{tabular}
\caption{
The braids $\delta$ and $\varepsilon$
correspond to rigid rotations of a punctured disk
when punctures are located as above.}
\label{fig:circ}
\end{figure}

The CDP for periodic braids is easy because an $n$-braid $\alpha$ is periodic if and only if
either $\alpha^n$ or $\alpha^{n-1}$ belongs to $\langle\Delta^2\rangle$
and because two periodic braids are conjugate if and only if
they have the same exponent sum.
The CSP for conjugates of $\delta^k$, $k\in\mathbb{Z}$, is also easy because
$[\delta^k]^S$ is the singleton $\{\delta^k\}$ in the dual Garside structure.
Therefore the main task for solving the conjugacy problem for periodic braids
is to solve the CSP for conjugates of $\varepsilon^k$.

\medskip

The present work is motivated by the paper~\cite{BGG07b} of
Birman, Gebhardt and Gonz\'alez-Meneses,
which provides a polynomial time solution
to the conjugacy problems for periodic braids.
Let us recall their results.

They first obtain a characterization of the elements
in the super summit set $[\varepsilon]^S$
in the \emph{classical} Garside structure.
This characterization is fruitful.
(i) It shows that the size of $[\varepsilon]^S$ is $(n-2)2^{n-3}$.
Because it is exponential in the braid index $n$,
the standard algorithms for the CSP (using the classical Garside structure)
applied to periodic $n$-braids are exponential in $n$.
(ii) It directly provides,
given any element $\alpha$ of $[\varepsilon]^S$,
an element $\gamma\in B_n$ such that $\gamma^{-1}\alpha \gamma=\varepsilon$.
This conjugating element $\gamma$ has canonical length at most 2.
(We remark that the statements in \cite{BGG07b} use ultra summit sets.
For periodic braids, the ultra summit set and the super summit set coincide.)

However, it is hard to extend this method to $\varepsilon^k$ for $k\geqslant 2$.
They overcame this difficulty by using several known isomorphisms
between certain subgroups of the braid groups.
By converting the CSP for conjugates of $\varepsilon^k$ in $B_n$
to the CSP for conjugates of $\delta^k$ in $B_{2n-2}$,
they gave a polynomial time solution to the CSP for conjugates of $\varepsilon^k$.
In this case, a conjugating element from an element $\alpha$
of $[\varepsilon^k]^S$ to $\varepsilon^k$
can be found only in an algorithmic way, not directly from $\alpha$
as in the case of $\varepsilon$.

\medskip

As we have seen above,
a nice characterization of the elements of $[\varepsilon^k]^S$
would give a better understanding of periodic braids.
We ask the following.

\medskip\noindent\textbf{Question.}\ \
Is there a characterization of the elements of $[\varepsilon^k]^S$,
which determines the size of $[\varepsilon^k]^S$
and directly provides a conjugating element
from any element of $[\varepsilon^k]^S$ to $\varepsilon^k$?
\medskip

In this paper, we give an affirmative answer to this question.
We give a characterization of the elements of $[\varepsilon^d]^S$
in the \emph{dual} Garside structure, where $d$ is
a divisor of $n-1$.

\medskip\noindent\textbf{Theorem \ref{thm:main1}.}\ \ \itshape
Let $B_n$ be endowed with the dual Garside structure.
Let $n=rd+1$ for integers $r\geqslant 2$ and $d\geqslant 1$.
Let $\alpha$ be a 1-pure braid in the super summit set of $\varepsilon^d$ in $B_n$.
If $d=1$, then $\alpha=\varepsilon$.
If $d\geqslant 2$, then
$\alpha=\delta^d a$ for a simple element $a$ of the form
$a=a_{d+1,1}b_0 b_1 \cdots b_{r-1}$, where
\begin{enumerate}
\item each $b_k$ is supported on $kd+\{2,3,\ldots,d+1\}$ for $0\leqslant k\leqslant r-1$, and

\item $b_0 \tau^{-d}(b_1)\tau^{-2d}(b_2)
\cdots\tau^{-(r-1)d}(b_{r-1})=[d+1,\ldots,3,2]$.
\end{enumerate}
In this case, $c^{-1}\alpha c=\varepsilon^d$
for a simple element $c$ given by
$$c=\tau^{-d}(b_1b_2\cdots b_{r-1})\tau^{-2d}(b_2b_3\cdots b_{r-1})\cdots
\tau^{-(r-1)d}(b_{r-1}).$$

\medskip
\upshape

In the above theorem, the decomposition $a=a_{d+1,1}b_0 b_1 \cdots b_{r-1}$ is unique,
and is easily obtained from any representation of $a$ in terms of generators of the braid group.

The characterization in the theorem determines the size of $[\varepsilon^d]^S$
as $n Z_d(\frac{n-1}d)=n{n-1\choose d-1}/d$,
where $Z_d(r)={rd\choose d-1}/d$ is the zeta polynomial
of the noncrossing partition lattice of $\{1,\ldots,d\}$.
(See Theorem~\ref{thm:size}.)
The characterization directly provides a conjugating element
from any element of $[\varepsilon^d]^S$ to $\varepsilon^d$.
This conjugating element has canonical length at most 1.

We remark that the condition that $d$ is a divisor of $n-1$
is a mild restriction.
Given any $k\in \mathbb{Z}$, let $d=\gcd(k, n-1)$.
Then $\varepsilon^k$ and $\varepsilon^d$ generate the same cyclic subgroup
in the central quotient $B_n/\langle\Delta^2\rangle$,
hence the CSP for conjugates of $\varepsilon^k$
is equivalent to the CSP for conjugates of $\varepsilon^d$.
(See Theorem~\ref{thm:e^d}.)

We hope that the method developed in this paper will be generalized to other groups
such as Artin groups of finite type and the braid groups
of complex reflection groups.

\section{Preliminaries}
In this section we recall the dual Garside structure
on the braid group $B_n$
and some known results on periodic braids.

\subsection{The dual Garside structure}

Birman, Ko and Lee~\cite{BKL98} introduced the following presentation for $B_n$:
$$
B_n  =  \left\langle a_{ij}, \ 1\leqslant j < i\leqslant n \left|
\begin{array}{ll}
a_{kl}a_{ij}=a_{ij}a_{kl}\quad \mbox{if $(k-i)(k-j)(l-i)(l-j)>0$}, \\
a_{ij}a_{jk}=a_{jk}a_{ik}=a_{ik}a_{ij}\quad \mbox{if $1\leqslant k<j<i \leqslant n$}.
\end{array}
\right.\right\rangle.
$$
The generators $a_{ij}$ are called {\em band generators}.
They are related to the classical generators by
$a_{ij}=\sigma_{i-1}\sigma_{i-2}\cdots\sigma_{j+1}\sigma_j
\sigma_{j+1}^{-1}\cdots\sigma_{i-2}^{-1}\sigma_{i-1}^{-1}$.
The {\em dual Garside structure} on $B_n$ is the triple $(B_n, B_n^+, \delta)$, where
$B_n^+$ is the monoid consisting of the elements
represented by positive words in the band generators, and
$\delta = a_{n,n-1}a_{n-1,n-2}\cdots a_{32} a_{21}$ is
the \emph{Garside element}.
(See \cite{DP99,Deh02} for details of Garside groups.)

There is a partial order $\preccurlyeq$ on $B_n$ defined as follows:
for $\alpha, \beta\in B_n$, $\alpha\preccurlyeq\beta$ if and only if
$\alpha^{-1}\beta\in B_n^+$.
Every pair of elements $\alpha, \beta\in B_n$ admits a unique lcm $\alpha\vee\beta$
and a unique gcd $\alpha\wedge\beta$
with respect to $\preccurlyeq$.
The set $[ 1, \delta ] =\{\alpha\in B_n : 1\preccurlyeq \alpha\preccurlyeq\delta\}$
generates $B_n$.
Its elements are called {\em simple elements}.
The expression $(1,\delta)$ denotes the set $[1,\delta]\setminus\{1,\delta\}$.

\medskip

For $\alpha\in B_n$, there are integers $r\leqslant s$ with
$\delta^r\preccurlyeq \alpha\preccurlyeq\delta^s$.
Hence the invariants
$\inf(\alpha)=\max\{r\in\mathbb{Z} : \delta^r\preccurlyeq \alpha\}$,
$\sup(\alpha)=\min\{s\in\mathbb{Z} : \alpha\preccurlyeq \delta^s\}$ and
$\operatorname{len}(\alpha)=\sup(\alpha)-\inf(\alpha)$
are well-defined. They are called
the \emph{infimum},
\emph{supremum} and
\emph{canonical length} of $\alpha$, respectively.
There exists a unique expression
$$
\alpha =\delta^r a_1\cdots a_\ell
$$
such that $r\in\mathbb{Z}$, $\ell\in\mathbb{Z}_{\geqslant 0}$, $a_1,\ldots,a_\ell\in (1,\delta )$ and
$(a_i a_{i+1})\wedge \delta= a_i$ for $i=1,\ldots,\ell-1$.
It is called the \emph{(left) normal form} of $\alpha$.
In this case, $\inf(\alpha)=r$ and\/ $\sup(\alpha)=r+\ell$,
so $\operatorname{len}(\alpha)=\ell$.

Let $\tau : B_n\to B_n$ be the inner automorphism
defined by $\tau(\alpha)=\delta^{-1}\alpha\delta$ for $\alpha\in B_n$.
Then $\tau(a_{ij})=a_{i+1,j+1}$, where indices are taken modulo $n$
and $a_{ji}=a_{ij}$ for $i>j$.
Hence $\tau$ preserves the monoid
$B_n^+$ and the set $[1,\delta]$ of simple elements.
The \emph{cycling} of $\alpha\in B_n$ is defined as
${\mathbf c}(\alpha) = \delta^r a_2\cdots a_\ell\tau^{-r}(a_1)$.

We denote by $[\alpha ]$ the conjugacy class
$\{\gamma^{-1}\alpha\gamma : \gamma\in B_n\}$.
Define ${\inf{\!}_s}(\alpha)=\max\{\inf(\beta):\beta\in [\alpha]\}$,
${\sup{\!}_s}(\alpha)=\min\{\sup(\beta): \beta\in [\alpha]\}$
and $\operatorname{len}_s(\alpha)=\min\{\,\operatorname{len}(\beta): \beta\in [\alpha]\,\}$.
The \emph{super summit set} $[\alpha]^S$,
the \emph{ultra summit set} $[\alpha]^U$
and the \emph{stable super summit set} $[\alpha]^{St}$
are subsets of the conjugacy class of $\alpha$
defined as follows:
\begin{align*}
[\alpha]^S &=\{\beta\in [\alpha]: \operatorname{len}(\beta)=\operatorname{len}_s(\alpha)\};\\{}
[\alpha]^U &=\{\beta\in [\alpha]^S:
\mathbf{c}^k(\beta)=\beta\ \mbox{ for some $k\geqslant 1$}\};\\{}
[\alpha]^{St} &=\{\beta\in [\alpha]^S: \beta^k\in[\alpha^k]^S \
\mbox{ for all $k\in{\mathbb Z}$}\}.
\end{align*}
It is known that $[\alpha]^S$, $[\alpha]^U$ and $[\alpha]^{St}$
are finite and nonempty~\cite{EM94,Geb05,BGG07a,LL08}.
It is also known that $\beta\in [\alpha]^S$
 if and only if
$\inf(\beta)={\inf{\!}_s}(\alpha)$ and $\sup(\beta)={\sup{\!}_s}(\alpha)$.

\subsection{Simple elements}
This subsection recalls some properties of simple elements
in the dual Garside structure of $B_n$.
See~\cite{BKL98,Bes03} for details.

\begin{definition}
For $T=\{i_1,\ldots,i_k\}$ with $1\leqslant i_1 < \cdots < i_k\leqslant n$,
the braid $a_{i_k i_{k-1}} \cdots a_{i_3 i_2} a_{i_{2}i_1}$
is denoted by either $b_T$ or $[i_k,\ldots,i_1]$.
Such a braid $b_T$ is called a \emph{subsimple}.
When $T$ is a singleton $\{i_1\}$, the braid
$b_{T}$ is defined to be the identity.
\end{definition}

A partition of a set is a collection of pairwise disjoint
subsets, called {\em blocks}, whose union is the entire set. A
partition of $\{1,\ldots,n\}$ is called a \emph{noncrossing
partition} if no two blocks cross each other, that is, if there is
no quadruple $(i_1,j_1,i_2,j_2)$ with $1\leqslant i_1<j_1<i_2<j_2\leqslant n$
such that $i_1$ and $i_2$ belong to one block and $j_1$ and $j_2$
belong to another. Equivalently, if $P_1,\ldots,P_n$ are vertices of
a regular $n$-gon as in the middle picture of Figure~\ref{fig:ncp},
then the convex hulls of
the respective blocks are pairwise disjoint.

\begin{figure}
\includegraphics[scale=1]{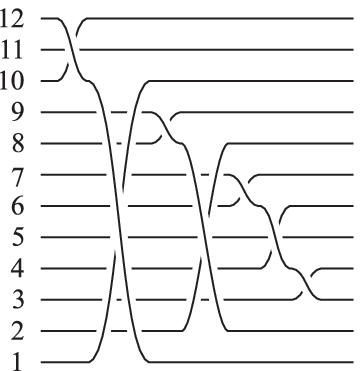}\qquad
\includegraphics[scale=.8]{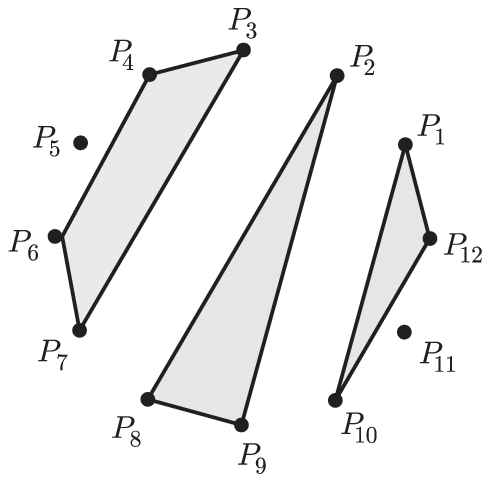}\qquad
\raisebox{12mm}{\footnotesize
$\begin{xy}
0;/r.20pc/:
(-5,0); (60,0) **@{-};
( 0,0); (45,0) **\crv{( 1,15) & (44,15)} ?(0.65);
( 0,0); (55,0) **\crv{( 1,19) & (54,19)} ?(0.65);
(45,0); (55,0) **\crv{(46, 5) & (54, 5)} ?(0.65);
( 5,0); (35,0) **\crv{( 6,11) & (34,11)} ?(0.65);
( 5,0); (40,0) **\crv{( 6,13) & (39,13)} ?(0.65);
(35,0); (40,0) **\crv{(36, 3) & (39, 3)} ?(0.65);
(10,0); (15,0) **\crv{(11, 3) & (14, 3)} ?(0.65);
(25,0); (30,0) **\crv{(26, 3) & (29, 3)} ?(0.65);
(15,0); (25,0) **\crv{(16, 5) & (24, 5)} ?(0.65);
(10,0); (30,0) **\crv{(11, 9) & (29, 9)} ?(0.65);
( 0,-3) *{1};
( 5,-3) *{2};
(10,-3) *{3};
(15,-3) *{4};
(20,-3) *{5};
(25,-3) *{6};
(30,-3) *{7};
(35,-3) *{8};
(40,-3) *{9};
(45,-3) *{10};
(50,-3) *{11};
(55,-3) *{12};
( 0,0) *{\bullet};
( 5,0) *{\bullet};
(10,0) *{\bullet};
(15,0) *{\bullet};
(20,0) *{\bullet};
(25,0) *{\bullet};
(30,0) *{\bullet};
(35,0) *{\bullet};
(40,0) *{\bullet};
(45,0) *{\bullet};
(50,0) *{\bullet};
(55,0) *{\bullet};
\end{xy}$}

\caption{The picture on the left is a braid diagram of the simple element
$[12,10,1]\,[9,8,2]\,[7,6,4,3]$ in $B_{12}$
and the other pictures illustrate the corresponding noncrossing partition
in two distinct ways.}
\label{fig:ncp}
\end{figure}

\begin{theorem}[Theorem 3.4 of ~\cite{BKL98}]
\label{thm:simple}
A braid $a\in B_n$ belongs to $[1, \delta]$ if and only if
$a=b_{T_1}\cdots b_{T_k}$ for some noncrossing
partition $\{T_1,\ldots,T_k\}$ of $\{1,\ldots,n\}$.
\end{theorem}

In the above theorem, for $a\in [1,\delta]$, 
the decomposition $a=b_{T_1}\cdots b_{T_k}$  
is unique up to reordering the factors because $b_{T_i}$ and $b_{T_j}$ commute for all $i,j$.
Each factor $b_{T_i}$ is called a \emph{subsimple of $a$}.

\begin{definition}
For a subsimple $b=[i_k,\ldots,i_1]$ with $k\geqslant 2$,
the set $\{i_1,\ldots,i_k\}$  is called the \emph{support} of $b$,
denoted by ${\operatorname{supp}}(b)$.
We say that a simple element $a$ is \emph{supported on}
$S\subseteq\{1,\ldots,n\}$
if ${\operatorname{supp}}(b)\subseteq S$ for each nonidentity subsimple $b$ of $a$.
\end{definition}

\begin{example}
Figure~\ref{fig:ncp} illustrates the noncrossing partition
$$\{\, \{12,10,1\},\ \{9,8,2\},\ \{7,6,4,3\},\ \{11\},\ \{5\}\, \}.$$
The corresponding simple element is $a=[12,10,1]\,[9,8,2]\,[7,6,4,3]$ in $B_{12}$.
Thus $a$ consists of three subsimples $[12,10,1]$, $[9,8,2]$ and $[7,6,4,3]$,
and it is supported on $\{1,\ldots,12\}\setminus\{5,11\}$.
\end{example}

For a noncrossing partition $\mathcal T=\{T_1,\ldots,T_k\}$ of $\{1,\ldots,n\}$,
we denote by $b_{\mathcal T}$ the braid $b_{T_1}\cdots b_{T_k}$.
Because two braids $b_{T_i}$ and $b_{T_j}$ commute for $i\ne j$,
the product $b_{T_1}\cdots b_{T_k}$
is well-defined.

By Theorem~\ref{thm:simple},
the set of simple elements in the dual Garside structure of $B_n$
is in one-to-one correspondence with
the set of noncrossing partitions of $\{1,\ldots,n\}$,
where the nonidentity subsimples correspond to the blocks containing at least two numbers.
For partitions $\mathcal T$ and $\mathcal T'$
of $\{1,\ldots,n\}$, it is known that
$b_{\mathcal T}\preccurlyeq b_{\mathcal T'}$
if and only if $\mathcal T$ is a refinement of $\mathcal T'$,
i.e.\ each block of $\mathcal T$ is a subset of some block of $\mathcal T'$.

\begin{lemma}\label{lem:subsim}
Let $b$ be a subsimple.
\begin{enumerate}
\item $a_{ij}\preccurlyeq b$ if and only if $\{i,j\}\subseteq{\operatorname{supp}}(b)$.
\item For a simple element $a$,  $a\preccurlyeq b$ if and only if $a$ is
supported on ${\operatorname{supp}}(b)$.
\end{enumerate}
\end{lemma}

The following lemmas provide easy ways to check
that a given positive braid is a simple element.

\begin{lemma}[Lemma 3.3 and Corollary 3.6 in~\cite{BKL98}]
\label{lem:simp2}
$a_{pq}a_{ij}$ is a simple element if and only if,
for $0<\lambda<1$,
we can connect $(p+\lambda,0)$ to $(q+\lambda,0)$
and $(i,0)$ to $(j,0)$ by nonintersecting arcs
in the upper half plane ${\mathbb R}^2_+=\{(x,y)\in{\mathbb R}^2\mid y\geqslant 0\}$.
(See Figure~\ref{fig:simp2}.)
\end{lemma}

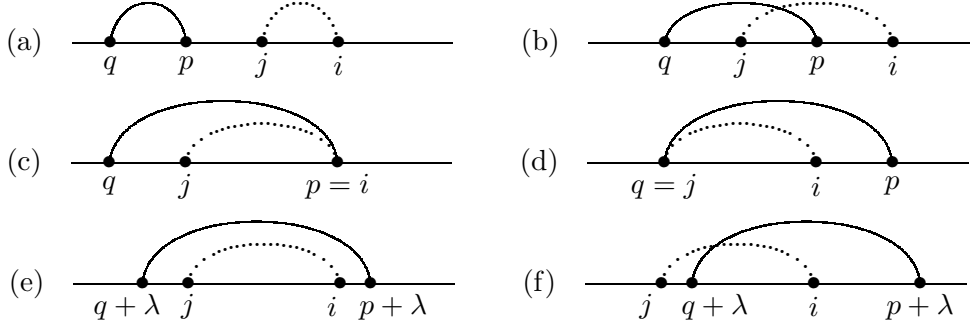
\begin{figure}
(a)\quad
$\begin{xy}
(5,0); (55,0) **@{-};
(10,0); (20,0) **\crv{(11,7) & (19,7)} ?(0.65) ;
(30,0); (40,0) **\crv{~*=<3pt>{.}  (31,7) & (39,7)} ?(0.65);
(10,0) *{\bullet};
(20,0) *{\bullet};
(30,0) *{\bullet};
(40,0) *{\bullet};
(10,-3) *{q};
(20,-3) *{p};
(30,-3) *{j};
(40,-3) *{i};
\end{xy}$
\qquad(b)\quad
$\begin{xy}
(0,0); (50,0) **@{-};
(10,0); (30,0) **\crv{(11,7) & (29,7)} ?(0.65) ;
(20,0); (40,0) **\crv{~*=<3pt>{.}  (21,7) & (39,7)} ?(0.65);
(10,0) *{\bullet};
(20,0) *{\bullet};
(30,0) *{\bullet};
(40,0) *{\bullet};
(10,-3) *{q};
(20,-3) *{j};
(30,-3) *{p};
(40,-3) *{i};
\end{xy}$

(c)\quad
$\begin{xy}
(5,0); (55,0) **@{-};
(20,0); (40,0) **\crv{~*=<3pt>{.} (21,7) & (39,7)} ?(0.65) ;
(10,0); (40,0) **\crv{(11,11) & (39,11)} ?(0.65) ;
(10,0) *{\bullet};
(20,0) *{\bullet};
(40,0) *{\bullet};
(10,-3) *{q};
(20,-3) *{j};
(40,-3) *{p=i};
\end{xy}$
\qquad(d)\quad
$\begin{xy}
(0,0); (50,0) **@{-};
(10,0); (30,0) **\crv{~*=<3pt>{.} (11,7) & (29,7)} ?(0.65) ;
(10,0); (40,0) **\crv{(11,11) & (39,11)} ?(0.65) ;
(10,0) *{\bullet};
(30,0) *{\bullet};
(40,0) *{\bullet};
(10,-3) *{q=j};
(30,-3) *{i};
(40,-3) *{p};
\end{xy}$

(e)\quad
$\begin{xy}
(5,0); (55,0) **@{-};
(20,0); (40,0) **\crv{~*=<3pt>{.} (21,7) & (39,7)} ?(0.65) ;
(14,0); (44,0) **\crv{(15,11) & (43,11)} ?(0.65) ;
(14,0) *{\bullet};
(20,0) *{\bullet};
(40,0) *{\bullet};
(44,0) *{\bullet};
(12,-3) *{q+\lambda};
(20,-3) *{j};
(39,-3) *{i};
(47,-3) *{p+\lambda};
\end{xy}$
\qquad(f)\quad
$\begin{xy}
(0,0); (50,0) **@{-};
(10,0); (30,0) **\crv{~*=<3pt>{.} (11,7) & (29,7)} ?(0.65) ;
(14,0); (44,0) **\crv{(15,11) & (43,11)} ?(0.65) ;
(10,0) *{\bullet};
(14,0) *{\bullet};
(30,0) *{\bullet};
(44,0) *{\bullet};
(8,-3) *{j};
(17,-3) *{q+\lambda};
(30,-3) *{i};
(44,-3) *{p+\lambda};
\end{xy}$

\caption{The pictures (a) and (c) show the case
where $a_{pq}a_{ij}$ is a simple element and
the pictures (b) and (d) show the other case.
The pictures (e) and (f) show the arcs
between $(p+\lambda,0)$ and $(q+\lambda,0)$
and between $(i,0)$ and $(j,0)$.}
\label{fig:simp2}
\end{figure}

\begin{lemma}[Lemma 3.3 and Corollary 3.6 in~\cite{BKL98}]
\label{lem:sim}
Let $a=x_1\cdots x_k\in B_n^+$ for some band generators $x_i$.
Then $a$ is a simple element if and only if
$x_ix_j$ is a simple element for all $1\leqslant i<j\leqslant k$.
\end{lemma}

An immediate consequence of the above lemma is the following.

\begin{lemma}\label{lem:defrag}
If $a_1b_1a_2b_2\cdots a_kb_k$ is a simple element
with $a_i, b_i\in B_n^+$ for $1\leqslant i\leqslant k$,
then $a_1a_2\cdots a_k$ is a simple element.
\end{lemma}

\subsection{Periodic braids}
This subsection recalls some results on periodic braids
in the dual Garside structure of $B_n$ from~\cite{LL11}.

\begin{lemma}[Lemma 3.6 in~\cite{LL11}] \label{lem:perlen}
If $\alpha\in B_n$ is periodic, then $\operatorname{len}_s (\alpha)\in\{0,1\}$.
\end{lemma}
By the above lemma, if $\alpha$ is a periodic braid then $[\alpha]^S=[\alpha]^U$.

The notion of cycling can be generalized as follows.

\begin{definition}\label{def:PartCycl}
Let $\alpha=\delta^u a_1a_2\cdots a_\ell \in B_n$ be in normal form.
Let $a_1'$ and $a_1''$ be simple elements such that $a_1=a_1'a_1''$.
The conjugation
$$
\tau^{-u}(a_1')^{-1} \alpha\, \tau^{-u}(a_1')
= \delta^u a_1''a_2\cdots a_\ell\,\tau^{-u}(a_1')
$$
is called a \emph{partial cycling} of $\alpha$ by $a_1'$.
\end{definition}

The following theorem shows that the conjugacy problem for conjugates of $\varepsilon^k$
is equivalent to that for conjugates of $\varepsilon^d$, where $d=\gcd(k,n-1)$,
and then it lists some properties of $[\varepsilon^d]^S$.
It comes from Lemmas 3.18, 3.21 and Theorem 3.24 in~\cite{LL11}.
\begin{theorem}\label{thm:e^d}
Let $B_n$ be endowed with the dual Garside structure.
\begin{enumerate}
\item
Let $k\ne 0$. Suppose $d$, $p$ and $q$ are integers
such that $d=\gcd(k,n-1)$ and $(n-1)p+kq=d$.
Let $\alpha\in B_n$ and $\alpha'=\delta^{np}\alpha^q$.
Then $\alpha$ is conjugate to $\varepsilon^k$ if and only if
$\alpha'$ is conjugate to $\varepsilon^d$.
Moreover, for $\gamma\in B_n$,  $\gamma^{-1}\alpha\gamma=\varepsilon^k$
if and only if $\gamma^{-1}\alpha'\gamma=\varepsilon^d$.

\item
Let $1\leqslant d<n-1$ be a divisor of\/ $n-1$, hence $n=rd+1$ for some $r\geqslant 2$.
\begin{enumerate}
\item $[\varepsilon^d]^{S}= [\varepsilon^d]^{U} = [\varepsilon^d]^{St}$;
\item $[\varepsilon^d]^{S}$ is closed under partial cyclings;
\item every element of $[\varepsilon^d]^{S}$ is of the form $\delta^d a$
for some $a\in(1,\delta)$ satisfying
$a\,\tau^{-d}(a)\,\tau^{-2d}(a)\cdots\tau^{-(r-1)d}(a)= \delta$.
\end{enumerate}
\end{enumerate}
\end{theorem}

\section{The super summit set of $\varepsilon^d$}

Throughout this section, we assume that the braid group $B_n$
is endowed with the dual Garside structure and that
$1\leqslant d<n-1$ is a divisor of $n-1$, hence $n=dr+1$ for some $r\geqslant 2$.

\begin{definition}
For a braid $\alpha\in B_n$, let $\pi_\alpha$ denote the induced
permutation of $\alpha$.
For $1\leqslant i\leqslant n$, a braid $\alpha\in B_n$ is called \emph{$i$-pure}
if $\pi_\alpha(i)=i$.
\end{definition}

For example, $\pi_\delta(i)=i-1\pmod n$ for $1\leqslant i\leqslant n$.
A noncentral $\varepsilon^k$ is $i$-pure
if and only if $i=1$.

\subsection{Typical super summit elements of $\varepsilon^d$}
\label{sec:example}

Let us consider typical elements of $[\varepsilon^d]^S$.
Recall that $n=rd+1$ with $r\geqslant 2$ and $d\geqslant 1$.
Suppose $d\ne 1$.

\begin{enumerate}
\item
Notice that $\varepsilon^d=\delta^d [d+1,\ldots,2,1]=\delta^d a_{d+1,1}[d+1,\ldots,2]$.
Let $c_0,c_1,\ldots,c_{r-1}$ be simple elements such that
$$[d+1,\ldots,2]=c_0c_1\cdots c_{r-1}.$$
See Figure~\ref{fig:e3}, where $n=13$, $d=3$ and $r=4$.
In this case, the simple elements $c_i$ are 13-braids supported on the set $\{2,3,4\}$.
Each {\fboxrule=.7pt\raisebox{3pt}{\fbox{$c_i$}}} in Figure~\ref{fig:e3}
(and also in Figures~\ref{fig:e3SS} and~\ref{fig:red} later)
means the 3-braid obtained from $c_i$ by deleting
all the $j$-th strands for $j\not\in\{2,3,4\}$.

\begin{figure}
\begin{tabular}{ccccc}
\includegraphics[scale=.9]{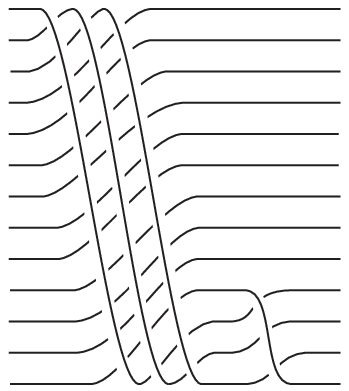}&\qquad&
\includegraphics[scale=.9]{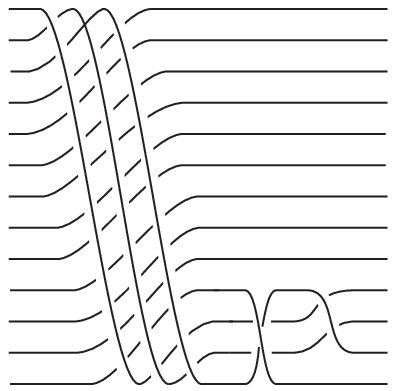}&\qquad&
\includegraphics[scale=.9]{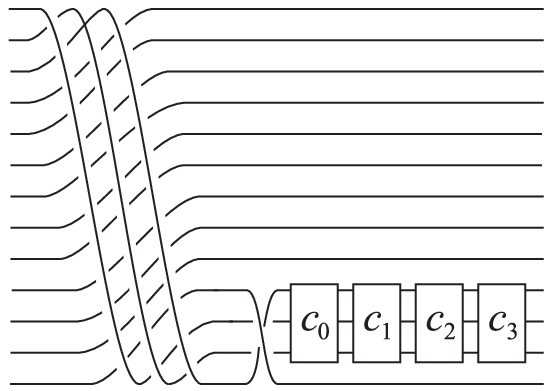}\\
(a) $\varepsilon^3=\delta^3[4,3,2,1]$ &&
(b) $\varepsilon^3=\delta^3a_{41}[4,3,2]$ &&
(c) $\varepsilon^3=\delta^3a_{41} c_0c_1c_2c_3$
\end{tabular}
\caption{Braid diagrams of $\varepsilon^3$ in $B_{13}$}
\label{fig:e3}
\end{figure}

\item
$\varepsilon^d=\delta^d a_{d+1,1}c_0c_1\cdots c_{r-1}$ is conjugate to
$$\alpha=\delta^d a_{d+1,1} b_0b_1\cdots b_{r-1},$$
where $b_k=\tau^{kd}(c_k)$ for $0\leqslant k\leqslant r-1$.
See Figure~\ref{fig:e3SS}(a).
If we let
$$
x=\tau^{-d}(b_1b_2\cdots b_{r-1})\tau^{-2d}(b_2b_3\cdots b_{r-1})\cdots
\tau^{-(r-1)d}(b_{r-1}),
$$
then $x^{-1}\alpha x=\varepsilon^d$.
Note that each $b_k$ is supported on $\{kd+2,\ldots,kd+(d+1)\}$.
For an example, see Figure~\ref{fig:e3SS}(b), where
$c_0=[3,2]$, $c_1=c_3=1$, $c_2=[4,2]$, hence
$b_0=[3,2]$, $b_1=b_3=1$, $b_2=[10,8]$.

\begin{figure}
\begin{tabular}{ccc}
\includegraphics[scale=.9]{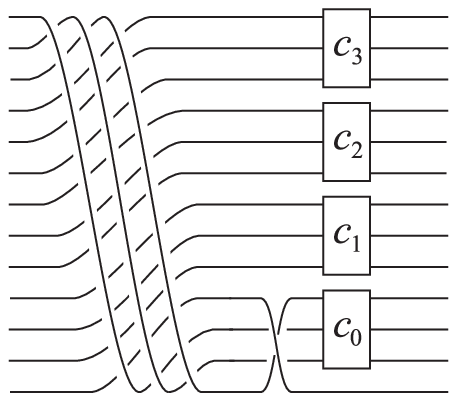}&\qquad&
\includegraphics[scale=.9]{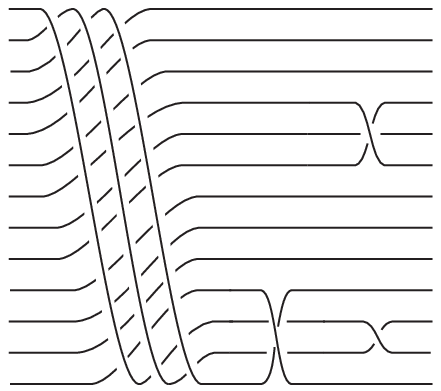}\\
(a) A 1-pure braid in $[\varepsilon^3]^S$&&
(b) $\delta^3 a_{41}[3,2][10,8]$\\
\end{tabular}
\caption{Elements of $[\varepsilon^3]^S$ in $B_{13}$.}
\label{fig:e3SS}
\end{figure}

\item Notice that $\alpha\in[\varepsilon^d]^S$ because
$\operatorname{len}(\alpha)=1= \operatorname{len}_s (\varepsilon^d)$.
\end{enumerate}

\subsection{Super summit elements of $\varepsilon^d$}

In this subsection, we give a characterization of the super summit
elements of $\varepsilon^d$.
For $0\leqslant k\leqslant r-1$, define
$$S_k=kd+\{2,3,\ldots,d+1\}=\{kd+2,kd+3,\ldots,kd+(d+1)\}.$$
Then $\{\{1\}, S_0, S_1,\ldots, S_{r-1}\}$ is a partition of $\{1,\ldots,n\}$.
The braid $\alpha$ in \S\ref{sec:example}
is 1-pure such that each $b_k$ is supported on $S_k$.

The following is the main theorem of this paper,
which shows that every 1-pure super summit element of $\varepsilon^d$
can be obtained using the construction in \S\ref{sec:example}.

\begin{theorem}\label{thm:main1}
Let $B_n$ be endowed with the dual Garside structure.
Let $n=rd+1$ for integers $r\geqslant 2$ and $d\geqslant 1$.
Let $\alpha$ be a 1-pure braid in the super summit set of $\varepsilon^d$ in $B_n$.
If $d=1$, then $\alpha=\varepsilon$.
If $d\geqslant 2$, then
$\alpha=\delta^d a$ for a simple element $a$ of the form
$a=a_{d+1,1}b_0 b_1 \cdots b_{r-1}$, where
\begin{enumerate}
\item each $b_k$ is supported on $kd+\{2,3,\ldots,d+1\}$ for $0\leqslant k\leqslant r-1$, and

\item $b_0 \tau^{-d}(b_1)\tau^{-2d}(b_2)
\cdots\tau^{-(r-1)d}(b_{r-1})=[d+1,\ldots,3,2]$.
\end{enumerate}
In this case, $c^{-1}\alpha c=\varepsilon^d$
for a simple element $c$ given by
$$c=\tau^{-d}(b_1b_2\cdots b_{r-1})\tau^{-2d}(b_2b_3\cdots b_{r-1})\cdots
\tau^{-(r-1)d}(b_{r-1}).$$
\end{theorem}

\begin{proof}
By Theorem~\ref{thm:e^d}, $\alpha=\delta^d a$ for some $a\in(1,\delta)$.

\begin{claim}{Claim 1}
$a_{d+1,1}\preccurlyeq a$
\end{claim}

\begin{proof}[Proof of Claim 1]
Since $\alpha=\delta^d a$ is 1-pure and $1=\pi_{\alpha}(1)=(\pi_{\delta^d}\circ\pi_a)(1)
=\pi_{a}(1)-d$, we have $\pi_a(1)=d+1$.
Hence there exists a subsimple $b$ of $a$ of the form
$b=[d+1,i_k,\ldots,i_1,1]$.
Since $[d+1,i_k,\ldots,i_1,1] = a_{d+1,1} [d+1,i_k,\ldots,i_1]$,
we have $a_{d+1,1}\preccurlyeq b\preccurlyeq a$.
\end{proof}

By the above claim, $a=a_{d+1,1}a'$ for a simple element $a'$.

If $d=1$, then $\alpha=\varepsilon$
because $\varepsilon=\delta a_{21}$
and $\alpha=\delta a_{21}a'$ have the same exponent sum.

Now we assume $d\geqslant 2$.

\begin{claim}{Claim 2}
$[rd+1, (r-1)d+1, \ldots,d+1,1]a'$ is a simple element.
\end{claim}

\begin{proof}[Proof of Claim 2]
By Theorem~\ref{thm:e^d}, $a\tau^{-d}(a)\cdots\tau^{-d(r-1)}(a)=\delta$.
Applying $\tau^{(r-1)d}$ to both sides, we have
$$\tau^{(r-1)d}(a)\cdots\tau^d(a)a=\delta.$$
Because $\tau^k(a)=\tau^k(a_{d+1,1})\tau^k(a')$ for each $k$,
the braid
$\tau^{(r-1)d} (a_{d+1,1})\cdots\tau^d(a_{d+1,1})a_{d+1,1} a'$
is a simple element by Lemma~\ref{lem:defrag}.
Since
\begin{align*}
\tau^{(r-1)d} & (a_{d+1,1})\cdots\tau^d(a_{d+1,1})a_{d+1,1}\\
&= a_{rd+1,(r-1)d+1} \cdots a_{2d+1,d+1} a_{d+1,1}\\
&= [rd+1,(r-1)d+1, \ldots,d+1,1],
\end{align*}
we are done.
\end{proof}

\begin{claim}{Claim 3}
Each subsimple of $a'$ is supported on $S_k$ for some $0\leqslant k\leqslant r-1$.
\end{claim}

\begin{proof}[Proof of Claim 3]
Assume not. By Lemma~\ref{lem:subsim},
there is a generator $a_{ij}\preccurlyeq a'$
such that $\{i,j\}$ is not contained in any $S_k$.
Then there exist $0\leqslant k<k'\leqslant r$ such that
$$j\leqslant kd+1<i\leqslant k'd+1.$$
There are four cases:
(a) $j<kd+1<i<k'd+1$;
(b) $j=kd+1<i=k'd+1$;
(c) $j=kd+1<i<k'd+1$;
(d) $j<kd+1<i=k'd+1$.
In any case,
$a_{k'd+1,kd+1}a_{ij}$ is not a simple element
by Lemma~\ref{lem:simp2}.
See Figure~\ref{fig:cdk}.

On the other hand, $a_{k'd+1,kd+1}\preccurlyeq [rd+1,(r-1)d+1,\ldots,d+1,1]$
by Lemma~\ref{lem:subsim}. Hence $a_{k'd+1,kd+1}a_{ij}$ is a simple element
by Claim 2 and Lemma~\ref{lem:defrag},
which is a contradiction.
\end{proof}

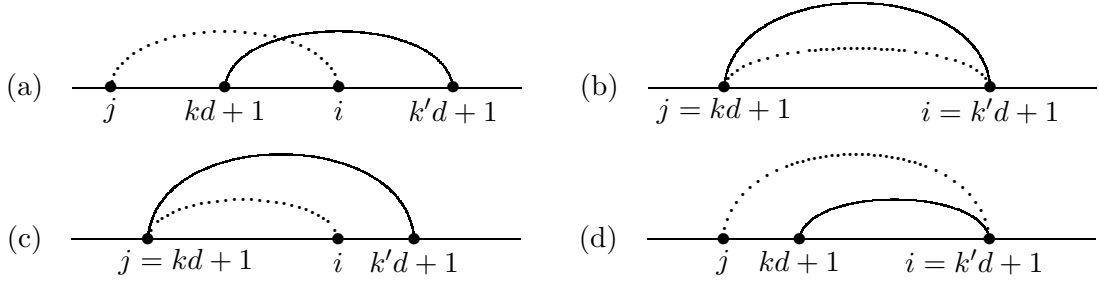
\begin{figure}
(a)\quad
$\begin{xy}
(0,0); (59,0) **@{-};
( 5,0); (35,0) **\crv{~*=<3pt>{.}  ( 6,10) & (34,10)} ?(0.65) ;
(20,0); (50,0) **\crv{(21,10) & (49,10)} ?(0.65);
( 5,0) *{\bullet};
(20,0) *{\bullet};
(35,0) *{\bullet};
(50,0) *{\bullet};
(5,-3) *{j};
(20,-3) *{kd+1};
(35,-3) *{i};
(50,-3) *{k'd+1};
\end{xy}$\qquad
(b)\quad
$\begin{xy}
(0,0); (59,0) **@{-};
(10,0); (45,0) **\crv{(11,15) & (44,15)} ?(0.65) ;
(10,0); (45,0) **\crv{~*=<3pt>{.}  (11,7) & (44,7)} ?(0.65);
(10,0) *{\bullet};
(45,0) *{\bullet};
(10,-3) *{j=kd+1};
(45,-3) *{i=k'd+1};
\end{xy}$

(c)\quad
$\begin{xy}
(0,0); (59,0) **@{-};
(10,0); (45,0) **\crv{(11,15) & (44,15)} ?(0.65) ;
(10,0); (35,0) **\crv{~*=<3pt>{.}  (11,7) & (34,7)} ?(0.65);
(10,0) *{\bullet};
(35,0) *{\bullet};
(45,0) *{\bullet};
(15,-3) *{j=kd+1};
(35,-3) *{i};
(45,-3) *{k'd+1};
\end{xy}$\qquad
(d)\quad
$\begin{xy}
(0,0); (59,0) **@{-};
(10,0); (45,0) **\crv{~*=<3pt>{.} (11,15) & (44,15)} ?(0.65) ;
(20,0); (45,0) **\crv{(21,7) & (44,7)} ?(0.65);
(10,0) *{\bullet};
(20,0) *{\bullet};
(45,0) *{\bullet};
(10,-3) *{j};
(20,-3) *{kd+1};
(43,-3) *{i=k'd+1};
\end{xy}$

\caption{$a_{k'd+1,kd+1} a_{ij}$ is not a simple element
in these four cases.}
\label{fig:cdk}
\end{figure}

For $0\leqslant k\leqslant r-1$, let $b_k$ be the product of
subsimples of $a'$ which are supported on $S_k$.
(Hence each $b_k$ is a simple element supported on $S_k$.)
By the above claim, we have
$$a=a_{d+1,1}b_0 b_1 \cdots b_{r-1}.$$

\begin{claim}{Claim 4}
$c=\tau^{-d}(b_1b_2\cdots b_{r-1})\tau^{-2d}(b_2b_3\cdots b_{r-1})\cdots
\tau^{-(r-1)d}(b_{r-1})$ is a simple element.
\end{claim}

\begin{proof}[Proof of Claim 4]
In the proof of Claim 2,
we have seen $a\tau^{-d}(a)\cdots\tau^{-(r-1)d}(a)=\delta$.
Then $c$ is a simple element by Lemma~\ref{lem:defrag} because $b_kb_{k+1}\cdots b_{r-1}\preccurlyeq a$
for $1\leqslant k\leqslant r-1$.
\end{proof}

We remark that $a_{d+1,1} b_0\tau^{-d}(b_1)\cdots\tau^{-kd}(b_k)$
is a simple element for $0\leqslant k\leqslant r-1$ by the same reason as above.

\begin{claim}{Claim 5}
$a_{d+1,1}b_0 \tau^{-d}(b_1)\tau^{-2d}(b_2)
\cdots\tau^{-(r-1)d}(b_{r-1})=[d+1,\ldots,2,1]$ and
$c^{-1}\alpha c=\varepsilon^d$.
\end{claim}

\begin{proof}[Proof of Claim 5]
Let $x_k=\tau^{-kd}(b_k b_{k+1}\cdots b_{r-1})$
and $\alpha_k=(x_1\cdots x_k)^{-1} \alpha (x_1\cdots x_k)$
for $1\leqslant k\leqslant r-1$.
Then $c=x_1x_2\cdots x_{r-1}$.

Since each $b_k$ is supported on $S_k$ for $1\leqslant k\leqslant r-1$,
it commutes with $a_{d+1,1}b_0$.
Since
\begin{align*}
\alpha &= \delta^d (a_{d+1,1}b_0)\cdot (b_1\cdots b_{r-1})
 = \delta^d (b_1\cdots b_{r-1}) \cdot (a_{d+1,1}b_0)\\
 &= \tau^{-d}(b_1\cdots b_{r-1})\cdot \delta^d a_{d+1,1}b_0
 = x_1 \cdot \delta^d a_{d+1,1}b_0,
\end{align*}
we have
\begin{align*}
\alpha_1 &=x_1^{-1}\alpha x_1 = \delta^d a_{d+1,1}b_0 x_1
 = \delta^d a_{d+1,1}b_0 \tau^{-d}(b_1\cdots b_{r-1})\\
 &=\delta^d a_{d+1,1}b_0\tau^{-d}(b_1)\cdot \tau^{-d}(b_2\cdots b_{r-1}).
\end{align*}
Notice that $a_{d+1,1}b_0\tau^{-d}(b_1)$ is supported on $\{1,2,\ldots,d+1\}$
and $\tau^{-d}(b_2\cdots b_{r-1})$ is supported on $S_1\cup\cdots \cup S_{r-2}$.
Therefore they commute.
By a similar computation as above,
$$
\alpha_2=x_2^{-1}\alpha_1x_2=\delta^d a_{d+1,1}b_0\tau^{-d}(b_1)\tau^{-2d}(b_2)\cdot \tau^{-2d}(b_3\cdots b_{r-1}).
$$
Continuing this computation, we obtain
$$
\alpha_{r-1}=x_{r-1}^{-1}\alpha_{r-2}x_{r-1}=\delta^d a_{d+1,1}b_0\tau^{-d}(b_1)\tau^{-2d}(b_2)\cdots \tau^{-(r-1)d}(b_{r-1}).
$$
Write $b=a_{d+1,1}b_0\tau^{-d}(b_1)\tau^{-2d}(b_2)\cdots \tau^{-(r-1)d}(b_{r-1})$.
Notice that $b$ is a simple element supported on $\{1,2,\ldots,d+1\}$,
hence $b\preccurlyeq [d+1,\ldots,2,1]$ by Lemma~\ref{lem:subsim}.
Because both elements have the same exponent sum, they must be equal, i.e.\
$b=[d+1,\ldots,2,1]$.
Hence
$\alpha_{r-1}=\delta^d[d+1,\ldots,2,1]=\varepsilon^d$.
Since $\alpha_{r-1}=(x_{r-1}^{-1}\cdots x_1^{-1})\alpha(x_1\cdots x_{r-1})
=c^{-1}\alpha c$,
we are done.
\end{proof}

Claims 1--5 complete the proof of Theorem~\ref{thm:main1}.
\end{proof}

Every element of $[\varepsilon^d]^S$ can be conjugated
to a 1-pure braid in $[\varepsilon^d]^S$ by $\delta^u$ for some $0\leqslant u\leqslant n-1$.
By Theorem~\ref{thm:main1} any 1-pure braid of $[\varepsilon^d]^S$
can be conjugated to $\varepsilon^d$ by a simple element $c$.
Therefore every element of $[\varepsilon^d]^S$ can be conjugated to
$\varepsilon^d$ by $\delta^u c$
for some $0\leqslant u\leqslant n-1$ and a simple element $c$.

\subsection{The size of the super summit set of $\varepsilon^d$}

Using the characterization of the elements of $[\varepsilon^d]^S$ in
Theorem~\ref{thm:main1},
we can figure out the size of $[\varepsilon^d]^S$.

Let $Z_d(r)$ be the number of multi-chains
$1\preccurlyeq a_1\preccurlyeq a_2\preccurlyeq\cdots\preccurlyeq a_{r-1}\preccurlyeq [d,d-1,\ldots,1]$.
$Z_d(r)$ is called the \emph{zeta polynomial
of the noncrossing partition lattice}
of $\{1,2,\ldots,d\}$~\cite{Ede80}.

\begin{theorem}[Edelman~\cite{Ede80}]
$Z_d(r)={dr\choose d-1}/d$ for $d,r\geqslant 1$.
\end{theorem}

An $r$-tuple $(c_0,\ldots,c_{r-1})$ of simple elements in $B_n$ is called
an \emph{$r$-composition} of a simple element $a$ if
$a=c_0\cdots c_{r-1}$.

There is a one-to-one correspondence between the set of $r$-compositions
of $[d,d-1,\ldots,1]$ and the set of multi-chains
$1\preccurlyeq a_1\preccurlyeq a_2\preccurlyeq\cdots\preccurlyeq a_{r-1}\preccurlyeq [d,d-1,\ldots,1]$.
For an $r$-composition $(c_0,\ldots,c_{r-1})$,
let $a_k=c_0c_1\cdots c_{k-1}$ for $1\leqslant k\leqslant r-1$.
Then $(a_1,\ldots,a_{r-1})$ is a multi-chain.
Conversely, for a multi-chain $(a_1,\ldots,a_{r-1})$,
let $c_k=a_k^{-1}a_{k+1}$ for $0\leqslant k\leqslant r-1$
where $a_0=1$ and $a_r=[d,d-1,\ldots,1]$.
Then $(c_0,c_1,\ldots,c_{r-1})$ is an $r$-composition.

\begin{theorem}\label{thm:size}
Let $n=rd+1$ for $r\geqslant 2$ and $d\geqslant 1$.
The cardinality of $[\varepsilon^d]^S$ is
$nZ_d(r)=n{dr\choose d-1}/d=n{n-1\choose d-1}/d$.
\end{theorem}

\begin{proof}
It suffices to show that $Z_d(r)$ is equal to the number of 1-pure braids in $[\varepsilon^d]^S$
because any element of $[\varepsilon^d]^S$ is uniquely expressed as
$\tau^u(\alpha)$ for $0\leqslant u\leqslant n-1$ and a 1-pure braid $\alpha$ in $[\varepsilon^d]^S$.

If $d=1$, then $Z_d(r)=1$ and $\varepsilon$ is the only 1-pure braid in $[\varepsilon]^S$
by Theorem~\ref{thm:main1}.
Now we assume $d\geqslant 2$.
Notice that there is a one-to-one correspondence between
$r$-compositions of $[d,d-1,\ldots,1]$ and those of $[d+1,d,\ldots,2]$.

Given an $r$-composition $(c_0, c_1,\ldots, c_{r-1})$ of $[d+1,d,\ldots,2]$, let
$$
\alpha=\delta^d a_{d+1,1}b_0b_1\cdots b_{r-1},
$$
where $b_k=\tau^{kd}(c_k)$ for $0\leqslant k\leqslant r-1$.
Then $\alpha$ is a 1-pure braid in $[\varepsilon^d]^S$
as we have seen in \S\ref{sec:example}.
Combining this observation with Theorem~\ref{thm:main1},
we can see that there is a one-to-one correspondence
between the set of 1-pure braids in $[\varepsilon^d]^S$ and the set
of $r$-compositions of $[d,d-1,\ldots,1]$.
Hence $Z_d(r)$ is equal to the number of 1-pure braids in $[\varepsilon^d]^S$.
\end{proof}

\subsection{The stable super summit sets of $\varepsilon^d$ and $\varepsilon^k$}

The CSP for conjugates of $\varepsilon^k$
is equivalent to the CSP for conjugates of $\varepsilon^d$
for $d=\gcd(k, n-1)$. (See Theorem~\ref{thm:e^d}.)
However, their super summit sets look different.
For example, $\varepsilon^k$ does not have such a nice
characterization of the super summit elements as $\varepsilon^d$.

For the study of conjugacy classes of periodic braids, the stable super summit set
looks more natural than the super summit set.
In Proposition~\ref{thm:St} we will see that the stable super summit set
of $\varepsilon^k$ is in one-to-one correspondence with that of $\varepsilon^d$.
Moreover, the correspondence is given by taking powers and by multiplying
by central elements.
In Corollary~\ref{thm:redu} we will see that all the 1-pure braids
in $[\varepsilon^k]^{St}$ have a common reduction system when $d\geqslant 2$.

\begin{proposition}\label{thm:St}
For an integer $k$, let $d=\gcd(k,n-1)$ and $(n-1)p+kq=d$.
Let $f:[\varepsilon^d]\to [\varepsilon^k]$ and $g:[\varepsilon^k]\to [\varepsilon^d]$
be the maps defined by
$$
f(\alpha)=\alpha^{k/d}\quad\mbox{and}\quad g(\beta)=\delta^{np}\beta^q,
$$
where $[\varepsilon^d]$ and $[\varepsilon^k]$ denote
the conjugacy classes of $\varepsilon^d$ and $\varepsilon^k$, respectively.

Then $f$ and $g$ are inverses to each other such that
$f([\varepsilon^d]^{St})=[\varepsilon^k]^{St}$
and
$g([\varepsilon^k]^{St})=[\varepsilon^d]^{St}$.
In particular,
$$ \#[\varepsilon^d]^S
=\#[\varepsilon^d]^{St}
=\#[\varepsilon^k]^{St}
\leqslant \#[\varepsilon^k]^S.
$$
\end{proposition}

\begin{proof}
First, observe that $f$ and $g$ are inverses to each other.
Let $\alpha$ be conjugate to $\varepsilon^d$.
Then $\alpha^{(n-1)/d}=\delta^n$
because $\alpha^{(n-1)/d}$ is conjugate to $(\varepsilon^d)^{(n-1)/d}=\varepsilon^{n-1}=\delta^n$
and $\delta^n$ is central. Hence
$$
g(f(\alpha))=\delta^{np}(\alpha^{\frac kd})^{q}
=\alpha^{\frac{p(n-1)}d}\alpha^{\frac{kq}d}
=\alpha^{\frac{(n-1)p+kq}d}
=\alpha.
$$

Let $\beta$ be conjugate to $\varepsilon^k$.
Then $\beta^{(n-1)/d}=\delta^{kn/d}$ because
$\beta^{(n-1)/d}$ is conjugate to
$\varepsilon^{k(n-1)/d}=(\varepsilon^{n-1})^{k/d}=(\delta^n)^{k/d}$
and $\delta^n$ is central. Hence
$$
f(g(\beta))
=(\delta^{np}\beta^q)^{\frac kd}
=\delta^{\frac{knp}d}\beta^{\frac{kq}d}
=\beta^{\frac{(n-1)p}d}\beta^{\frac{kq}d}
=\beta^{\frac{(n-1)p+kq}d}
=\beta.
$$

Notice that, for $\alpha\in B_n$,  if $\beta\in[\alpha]^{St}$,
then $\beta^m\in[\alpha^m]^{St}$ and $\delta^{mn}\beta\in[\delta^{mn}\alpha]^{St}$
for every $m\in{\mathbb Z}$ .
Thus $f([\varepsilon^d]^{St})\subseteq [\varepsilon^k]^{St}$
and $g([\varepsilon^k]^{St})\subseteq [\varepsilon^d]^{St}$.
Then it follows that  $f([\varepsilon^d]^{St})=[\varepsilon^k]^{St}$
and $g([\varepsilon^k]^{St})=[\varepsilon^d]^{St}$ because
$f\circ g$ and $g\circ f$ are the identity functions on the conjugacy classes
$[\varepsilon^k]$ and $[\varepsilon^d]$, respectively.
In particular, $\#[\varepsilon^d]^{St}=\#[\varepsilon^k]^{St}$.
Because $[\varepsilon^d]^S=[\varepsilon^d]^{St}$ by Theorem~\ref{thm:e^d}
and $[\varepsilon^k]^{St}\subseteq [\varepsilon^k]^S$ by definition,
we are done.
\end{proof}

\begin{example}
We give an example showing that $[\varepsilon^k]^S$ can be strictly larger than
$[\varepsilon^k]^{St}$.
Consider $\varepsilon=\delta[2,1]$ and
$\varepsilon^2=\delta^2[3,2,1]=\delta^2[2,1][3,1]$ in $B_6$.
Let
$$\alpha=[3,1]\varepsilon^2 [3,1]^{-1}=\delta^2[5,3][2,1].$$
Then $\alpha\in[\varepsilon^2]^S$ as 
$\operatorname{len}(\alpha)=1=\operatorname{len}_s (\varepsilon^2)$.
But $\alpha^2=\delta^4[5,4,3,1]\cdot[2,1]\not\in[\varepsilon^4]^S$ as
$\operatorname{len}(\alpha^2)=2\ne 1= \operatorname{len}_s (\varepsilon^4)$.
Therefore $\alpha\in [\varepsilon^2]^S\setminus [\varepsilon^2]^{St}$.
\end{example}

Let $D_n\subseteq \mathbb C$ be the closed disk of radius $n+1$, centered at the origin,
with punctures at $\{1,2,\ldots, n\}$.
The $n$-braids are identified with the mapping classes of $D_n$.
A braid $\alpha$ is called \emph{reducible} if it preserves setwise a family of
nondegenerate simple closed curves, called a \emph{reduction system} for $\alpha$.
A reduction system is called \emph{round}
if each component is homotopic to a geometric circle
in $D_n$.

If $d=\gcd(k,n-1)\geqslant 2$, then the periodic braids $\varepsilon^d$ and $\varepsilon^k$
are reducible with a round reduction system.
See Figure~\ref{fig:red}(a) for a braid diagram of $\varepsilon^3\in B_{13}$.
Figure~\ref{fig:red}(c) illustrates a round reduction system for $\varepsilon^3$.
Figure~\ref{fig:red}(b) is another braid diagram for the 1-pure braid
of $[\varepsilon^3]^S$ in Figure~\ref{fig:e3SS}(a),
from which we can see that Figure~\ref{fig:red}(c) is also a reduction system
for this braid.
This happens for any 1-pure braid of $[\varepsilon^k]^{St}$
by Theorem~\ref{thm:main1} and Proposition~\ref{thm:St}.

\begin{corollary}\label{thm:redu}
If $d=\gcd(k,n-1)\geqslant 2$, then every 1-pure braid in $[\varepsilon^k]^{St}$
has a round reduction system,
consisting of $(n-1)/d$ circles each of which
encloses $d$ punctures, as in Figure~\ref{fig:red}(c).
\end{corollary}

\begin{figure}
\begin{tabular}{ccccc}
\includegraphics[scale=.8]{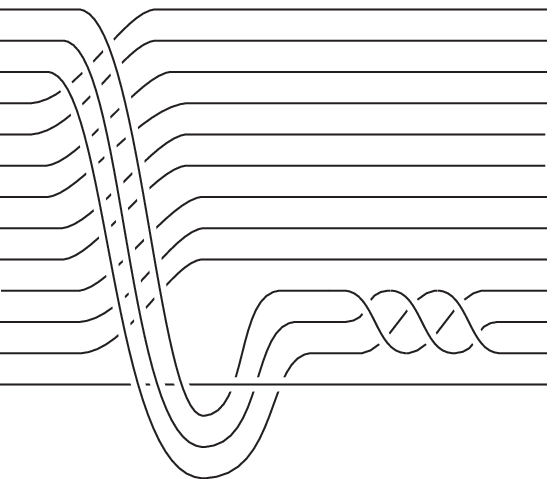}&\qquad&
\includegraphics[scale=.8]{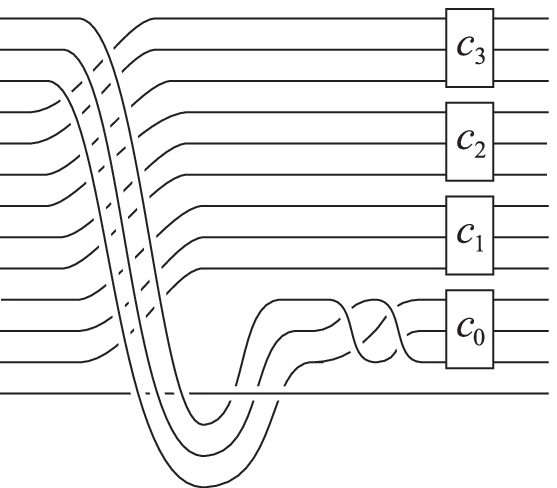}&\qquad&
\includegraphics[scale=.8]{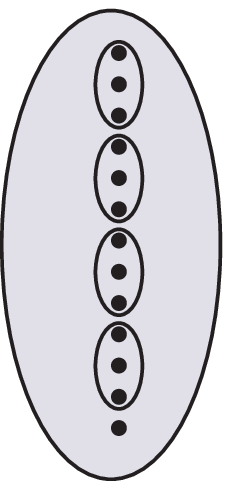}\\
(a) $\varepsilon^3\in B_{13}$  &&
(b) A 1-pure braid in $[\varepsilon^3]^S$ &&
(c) A reduction system
\end{tabular}
\caption{The braid diagrams (a) and (b)
show that 1-pure braids in $[\varepsilon^3]^S=[\varepsilon^3]^{St}$
are reducible with a \emph{round} reduction system shown in (c).}
\label{fig:red}
\end{figure}

\section*{Acknowledgements}
The authors are grateful to the anonymous referees for their careful reading and helpful comments.
The first author was partially supported by NRF-2015R1C1A2A01051589.
The second author was partially supported by NRF-2015R1D1A1A01056723.

\end{document}